\documentclass[12pt,a4paper]{article}
\usepackage[utf8]{inputenc}
\usepackage{amsmath}
\usepackage{amsfonts}
\usepackage{amssymb}
\usepackage{mathtools}
\usepackage{amsthm}
\usepackage{authblk}
\usepackage{graphicx}
\usepackage{caption}
\usepackage{subcaption}
\newtheorem{theorem}{Theorem}

\newtheorem{lemma}{Lemma}

\newtheorem*{proof*}{Proof}

\begin{document}
\title{Local and global results for modified Sz\'{a}sz - Mirakjan operators}
\author[1, 2, a]{R.B. Gandhi} \author[1, 3, b, *]{Vishnu Narayan Mishra}
\affil[1]{\scriptsize Applied Mathematics \& Humanities Department, Sardar Vallabhbhai National Institute of Technology, Ichchhanath Mahadev, Dumas Road, Surat - 395 007, (Gujarat), India}
\affil[2]{\scriptsize Department of Mathematics, BVM Engineering College, Vallabh Vidyanagar - 388 120, (Gujarat), India}
\affil[3]{\scriptsize L. 1627 Awadh Puri Colony Beniganj, Phase-III, Opposite - Industrial Training Institute (ITI), Ayodhya Main Road, Faizabad, Uttar Pradesh 224 001, India}
\date{}

\maketitle

\let\thefootnote\relax\footnotetext{\scriptsize E-mails: rajiv55in@yahoo.com$^{a}$; vishnunarayanmishra@gmail.com$^{b}$  \\ \scriptsize $^{*}$Corresponding author}

\begin{abstract}
\noindent \scriptsize In this paper, we study a natural modification of Sz\'{a}sz - Mirakjan operators. It is shown by discussing many important established results for Sz\'{a}sz - Mirakjan operators. The results do hold for this modification as well, be they local in nature or global, be they qualitative or quantitative. It is also shown that this generalization is meaningful by means of examples and graphical representations. \\ \\



\end{abstract}

\section{Introduction}
Mishra et al. \cite{RBG} introduced Sz\'{a}sz-Mirakjan-Durrmeyer-type generalization given by 
\begin{equation}\label{RBG}
D_n^{*}(f;x) = b_n\sum_{k=0}^{\infty} s_{b_n,k}(x)\int_0^{\infty}s_{b_n,k}(t)f\left(t\right)\, dt,
\end{equation}
where 
\begin{equation}\label{Sza}
s_{b_n,k}(x)=e^{-b_nx}\dfrac{(b_nx)^k}{k!}, \quad k = 0,1,2, \ldots; n \in \mathbb{N},
\end{equation}
$(b_n)_{1}^{\infty}$ is an increasing sequence of positive real numbers, $ b_n \rightarrow \infty$ as $n \rightarrow \infty$, $b_1 \geq 1$ and studied the simultaneous approximation properties of the operators (\ref{RBG}). References \cite{RBG1} and \cite{RBG2} contains some more work in this direction.

This type of generalization was introduced by Durrmeyer \cite{Dur} generalizing Bernstein polynomials by introducing both the summation and integration processes and introduced the summation-integral type approximation process, using the Bernstein polynomials, as follows:
\begin{equation}\label{Dur}
D_n(f;x)=(n+1)\sum_{k=0}^n b_{n,k}(x)\left(\int_{0}^1 b_{n,k}(t)f(t)\,dt \right),
\end{equation}
where 
\begin{equation*}
b_{n,k}(x)= \left(\begin{array}{c} n \\ k  \end{array} \right)x^k (1-x)^{n-k}, \quad k=0,1, \ldots, n,
\end{equation*}
$b_{n,k}$'s are called the Bernstein basis functions.
\\
Derriennic \cite{Der} proved several results concerning these operators, including the approximation of the $r^{\text{th}}$-derivative of a function by operators $D_n$. 
\\
Sz\'{a}sz \cite{Sza} and Mirakjan \cite{Mir} introduced and studied operators on unbounded interval $[0, \infty)$, known as Sz\'{a}sz-Mirakjan operators given by 
\begin{equation}\label{Mir}
M_n(f;x)= \sum_{k=0}^{\infty}s_{n,k}(x)f\left(\dfrac{k}{n}\right),
\end{equation}
where
\begin{equation}\label{Sza1}
s_{n,k}(x)=e^{-nx}\dfrac{(nx)^k}{k!}, \quad k = 0,1,2, \ldots; n \in \mathbb{N}.
\end{equation}
Here $s_{n,k}$'s are known as Sz\'{a}sz basis functions. The operators $\left(M_n \right)_{1}^{\infty}$ were extensively studied in 1950 by O. Sz\'{a}sz \cite{Sza}. 
\\
Mazhar and Totik \cite{Maz} introduced two Durrmeyer type 
modifications of Sz\'{a}sz-Mirakjan operators (\ref{Mir}), which are defined on unbounded interval $[0, \infty)$ as
\begin{equation*}
T^{*}_n(f;x)=f(0)s_{n,0}(x)+ n\sum_{k=1}^{\infty}s_{n,k}(x)\int_{0}^{\infty}s_{n,k-1}(t)f(t)\,dt
\end{equation*}
and
\begin{equation}\label{Maz1}
T_n(f;x)= n\sum_{k=0}^{\infty}s_{n,k}(x)\int_{0}^{\infty}s_{n,k}(t)f(t)\,dt,
\end{equation}
where $s_{n,k}$'s are as given by (\ref{Sza1}).
\\
Operators in (\ref{RBG}) are generalization of the operators (\ref{Maz1}) by means of the introduction of sequence $(b_n)_{1}^{\infty}$ for $n$. This introduction of sequence $(b_n)_{1}^{\infty}$ is very natural and it is shown in this paper by introducing it on operators (\ref{Mir}) and considering the operators, denoted by $S_n$, defined as follows:
\begin{equation}\label{RBG1}
S_n(f;x)= \sum_{k=0}^{\infty}s_{b_n,k}(x)f\left(\dfrac{k}{b_n}\right),
\end{equation}
where $s_{b_n,k}$'s are as given in (\ref{Sza}), $(b_n)_1^{\infty}$ is an increasing sequence of positive real numbers, $ b_n \rightarrow \infty$ as $n \rightarrow \infty$, $b_1 \geq 1$. Clearly, for $b_n = n $, we get (\ref{Mir}).
\\
Walczak, in \cite{Wal1} introduced a generalization of (\ref{Mir}) given as follows:
\begin{equation}\label{Wal}
S_n[f; a_n, b_n, q, x] := \sum_{k=0}^{\infty}s_{a_n,k}(x)f\left(\dfrac{k}{b_n + q}\right),
\end{equation}
where $q \geq 0$ is a fixed number, $(a_n)_1^{\infty}$ and $(b_n)_1^{\infty}$ are given increasing and unbounded numerical sequences such that $b_n \geq a_n \geq 1$, and $(a_n/b_n)_1^{\infty}$ is non-decreasing and 
\begin{equation*}
\dfrac{a_n}{b_n} = 1 + o\left(\dfrac{1}{b_n} \right).
\end{equation*}
It can be observed that for $a_n = b_n = n$ and $q = 0$, (\ref{Wal}) reduces to (\ref{Mir}) and for $a_n = b_n$ and $q = 0$, it reduces to (\ref{RBG1}). Walczak \cite{Wal1} discussed direct results related to pointwise and uniform convergence of the operators (\ref{Wal}) in exponential weight spaces.
\\
In Section $2$ of the paper, we discuss some results related to the operators (\ref{RBG1}) while operating upon test functions and calculations of moments for the operators (\ref{RBG1}) is carried out. In Section $3$, local results related to the operators (\ref{RBG1}) are derived. The global properties of the operators (\ref{RBG1}) are derived in Section $4$ in polynomial weight spaces where a characterization is established. It can be seen that these direct and indirect results are natural extensions of the results derived in \cite{Bec} for (\ref{Mir}) in polynomial weight spaces. Section $5$ discusses examples of the sequences $(b_n)_1^{\infty}$ and the corresponding approximate graphical representations under the operators (\ref{Mir}) and (\ref{RBG1}). 

\section{Elementary results}
\subsection{Estimation of moments}
To begin with, we give some auxiliary results. 
\begin{lemma} \label{L3}
For $e_i(t)=t^i, \quad i = 0,1,2,3,4$, the following holds:
\begin{enumerate}
\item[(a)]$S_n(e_0;x)=1,$
\item[(b)]$S_n(e_1;x)=x,$
\item[(c)]$S_n(e_2;x)=\dfrac{1}{b_n}(b_nx^2+x),$
\item[(d)]$S_n(e_3;x)=\dfrac{1}{b_n^2}(b_n^2x^3+3b_nx^2+x),$
\item[(e)]$S_n(e_4;x)=\dfrac{1}{b_n^3}(b_n^3x^4 + 6b_n^2x^3+7b_nx^2+x).$
\end{enumerate}
\end{lemma}
\begin{proof}
By elementary calculations, the results can be obtained.
\end{proof}
Remarks: 
\begin{enumerate}
\item[1.] From lemma \ref{L3}, it can be seen that the operators (\ref{RBG1}) preserves linearity. 
\item[2.] We have 
\begin{eqnarray*}
s_{b_n, k}^{'}(x) &=& -b_n e^{-b_nx} \dfrac{(b_nx)^k}{k!}+ e^{-b_nx}b_n \dfrac{k(b_nx)^{k-1}}{k!} \\ &=& -b_n s_{b_n, k}(x)+ \dfrac{k}{x}s_{b_n, k}(x),  
\end{eqnarray*}
hence
\begin{eqnarray} \label{e21}
\dfrac{x}{b_n} s_{b_n, k}^{'}(x) &=& \left( \dfrac{k}{b_n} - x \right)s_{b_n, k}(x).  
\end{eqnarray}
\item[3.] We have for $r \in \mathbb{N}$,
\begin{eqnarray*}
 S_n^{'}(t^r;x) &=& \sum_{k=0}^{\infty} s_{b_n, k}^{'}(x) \left(\dfrac{k}{b_n}\right)^{r} \\ &=& \sum_{k=0}^{\infty} \dfrac{b_n}{x}\left( \dfrac{k}{b_n} - x \right)s_{b_n, k}(x) \left(\dfrac{k}{b_n}\right)^{r} \qquad \text{(using (\ref{e21}) } 
\end{eqnarray*}
On re-arranging terms, we get
\begin{equation} \label{e22}
S_n(t^{r+1};x) = \dfrac{x}{b_n} S_n^{'}(t^r;x) + xS_n(t^r;x)
\end{equation}

\end{enumerate}
\begin{lemma} \label{L4}
For $r \in \mathbb{N}$, the following relation holds:
\begin{equation} \label{e23}
S_n(t^r;x) = \sum_{j=1}^r a_{r,j}x^jb_n^{j-r} = x^r + \dfrac{r(r-1)}{2b_n}x^{r-1}+ \cdots + b_n^{1-r}x
\end{equation}
with positive coefficients $a_{r, j}$. In particular, $S_n(t^r;x)$ is a polynomial of degree $r$ without a constant term. 
\end{lemma}
\begin{proof}
From (b)-(c) of lemma \ref{L3}, the representation (\ref{e23}) holds true for $r = 1, 2$ with $a_{1,1}=a_{2,2}=a_{2,1}=1$. We use the principle of mathematical induction and assume (\ref{e23}) to be true for some positive integer $r$. Now from (\ref{e22}),
\begin{eqnarray*}
S_n(t^{r+1};x) &=& \dfrac{x}{b_n} S_n^{'}(t^r;x) + xS_n(t^r;x) \\ &=& \dfrac{x}{b_n}\sum_{j=1}^{r}ja_{r,j}x^{j-1}b_n^{j-r}+ x\sum_{j=1}^{r}a_{r,j}x^{j}b_n^{j-r} \\ &=& a_{r,1}b_n^{-r}x + \sum_{j=2}^{r}\left(ja_{r,j}+ a_{r, j-1}\right) x^{j}b_n^{j-(r+1)}+a_{r,r}x^{r+1} \\ &=: &\sum_{j=1}^{r+1}a_{r+1,j}x^{j}b_n^{j-(r+1)},
\end{eqnarray*}
say. Thus the representation (\ref{e23}) is valid for all $r \in \mathbb{N}$ since
\begin{eqnarray*}
 && a_{r+1,r+1}=a_{r,r}= \cdots = a_{1,1}=1, \\ &&
 a_{r+1,1}=a_{r,1}= \cdots = a_{1,1}=1, \\ &&
 a_{r+1,r}=ra_{r,r} + a_{r,r-1} = r + a_{r,r-1}\\&& \qquad \quad = \cdots = r + (r-1)+ \cdots + 2 + a_{2,1} = r(r+1)/2.
\end{eqnarray*}
\end{proof}
Remark: In connection with the coefficients $a_{r, r-1}$, it can be derived that for $N \geq 2$,
\begin{equation}\label{e24}
a_{N+2, N+1}-2a_{N+1, N}+a_{N, N-1} = 1.
\end{equation}
Let us denote by $\mu_{n,m}$, the $m^{th}$ moments of the operators given by (\ref{RBG1}), defined as 
\begin{equation}\label{E5}
\mu_{n,m}(x)= S_n((t-x)^m;x),\qquad m = 0,1,2,\ldots. 
\end{equation}

\begin{lemma}\label{L5}
For the moments defined in ($\ref{E5}$), the following holds:
\begin{enumerate}
\item[(a)]$\mu_{n,1}(x) = S_n((t-x);x)=0 ,$
\item[(b)]$\mu_{n,2}(x)= S_n((t-x)^2;x)=\dfrac{x}{b_n},$
\item[(c)]$\mu_{n,3}(x)= S_n((t-x)^3;x)=\dfrac{x}{b_n^2},$
\item[(d)]$\mu_{n,4}(x)= S_n((t-x)^4;x)=\dfrac{3x^2}{b_n^2}+\dfrac{x}{b_n^3}.$
\end{enumerate}
\end{lemma}
\begin{proof} 

The results follow from linearity of the operators $S_n$ and lemma $\ref{L3}$.
\end{proof}
Also, we mention some results related to the (modified) Steklov means. The (modified) Steklov means for $h > 0$ is defined by,
\begin{equation*}
f_h(x):= \left(\dfrac{2}{h}\right)^2 \int_0^{h/2} \int_0^{h/2} [2f(x+s+t) - f(x + 2(s+t))]ds dt.
\end{equation*}
We have
\begin{eqnarray*}
f(x) - f_h(x)&=& \left(\dfrac{2}{h}\right)^2 \int_0^{h/2} \int_0^{h/2} \Delta_{s+t}^2 f(x) ds dt, \\ f_h^{''}(x) &=& h^{-2}\left[8\Delta_{h/2}^2 f(x) - \Delta_{h}^2 f(x) \right],
\end{eqnarray*}
and hence
\begin{equation} \label{e3}
\Vert f - f_h \Vert_N \leq \omega_N^2(f, h), \qquad \Vert f_h^{''} \Vert_N \leq 9 h^{-2}\omega_N^2(f, h).
\end{equation}

\section{Local results}
\subsection{Direct result}
\noindent Consider the Banach lattice
\begin{equation*}
C_\gamma[0,\infty) = \{f \in C[0,\infty): \left|f(t) \right| \leq M (1+t)^{\gamma}\}
\end{equation*}
for some $M>0, \gamma>0$.

\begin{theorem}\label{T1}
$\lim_{n \rightarrow \infty} S_n(f;x)= f(x)$ uniformly for $x \in [0,a]$, provided $f \in C_\gamma [0,\infty)$, $\gamma \geq 2$ and $a > 0$.
\end{theorem}
\begin{proof} 
For fix $a >0$, consider the lattice homomorphism $T_a: C[0, \infty) \rightarrow C[0,a]$ defined by $T_a(f):= \left.f\right|_{[0,a]}$ for every $f \in C[0, \infty)$, where $\left.f\right|_{[0,a]}$ denotes the restriction of the domain of $f$ to the interval $[0,a]$. In this case, we see that, for each $i = 0,1,2$ and by (a)-(c) of lemma \ref{L3},

\begin{equation}\label{E13}
\lim_{n \rightarrow \infty} T_a \left(S_n(e_i;x)\right)= T_a( e_i(x)), \; \text{uniformly on} \;[0,a].
\end{equation}

Thus, by using (\ref{E13}) and with the universal Korovkin-type property with respect to  positive linear operators (see Theorem 4.1.4 (vi) of \cite{Alt}, p.199) we have the result.
\end{proof}

Let us consider the space $C_B[0,\infty)$ of all continuous and bounded functions on $[0, \infty)$ and for $f \in C_B[0,\infty)$, consider the supremum norm $\|f\|=\sup\{|f(x)|: x \in [0,\infty)\} $. Also, consider the $K-$functional
\begin{equation}\label{E7}
K_2(f;\delta)=\inf_{g \in W^2}\left\lbrace \|f-g \|+\delta\|g^{''} \| \right\rbrace,
\end{equation}
where $\delta>0$ and $W^2=\left\lbrace g \in C_B[0,\infty): g^{'}, g^{''} \in C_B[0, \infty) \right\rbrace$. For a constant $C >0$, the following relationship exists:
\begin{equation}\label{E8}
K_2(f;\delta)\leq C \omega_2(f, \sqrt{\delta}),
\end{equation}
where
\begin{equation}\label{E9}
\omega_2(f, \sqrt{\delta})=\sup_{0<h<\sqrt{\delta}} \; \sup_{x \in [0, \infty)}|f(x+2h)-2f(x+h)+f(x)|
\end{equation}
is the second order modulus of smoothness of $f \in C_B[0,\infty)$; and for $f \in C_B[0,\infty)$, let the modulus of continuity be given by
\begin{equation}\label{E10}
\omega_1(f, \sqrt{\delta})=\sup_{0<h<\sqrt{\delta}} \; \sup_{x \in [0, \infty)}|f(x+h)-f(x)|.
\end{equation}

\begin{theorem}\label{T2}
For $f \in C_B[0, \infty)$, we have
\begin{equation*}
|S_n(f;x) - f(x)| \leq C\omega_2 \left(f,\sqrt{\mu_{n,2}(x)}\right),
\end{equation*}
where $C$ is a positive constant. 
\end{theorem}

\begin{proof}
For $g \in W^2$, $x \in [0, \infty)$ and by Taylor's expansion, we have
\begin{equation*}
g(t) = g(x) + (t-x)g^{'}(x)+\int_x^t(t-u)g^{''}(u)du.
\end{equation*}
Operating $S_n$ on both the sides, 
\begin{eqnarray*}
\left|S_n(g;x)-g(x)\right| &=&\left| S_n\left( \int_x^t(t-u)g^{''}(u)du;x\right)\right| \\ &\leq & \dfrac{1}{2}\Vert g^{''} \Vert S_n((t-x)^2;x) \\ & = & \dfrac{1}{2} \Vert g^{''} \Vert  \mu_{n,2}(x).
\end{eqnarray*}
Also, we have $\left|S_n(f;x)\right| \leq \Vert f \Vert$. Using these, we get
\begin{eqnarray*}
\left|S_n(f;x)- f(x)\right|& \leq & \left|S_n(f-g;x)- (f-g)(x)\right| + \left|S_n(g;x)- g(x)\right| \\ & \leq & 2 \Vert f-g \Vert + \dfrac{1}{2} \Vert g^{''} \Vert \mu_{n,2}(x).
\end{eqnarray*}
Taking infimum on the right hand side for all $g \in W^2$, we get
\begin{eqnarray*}
\left|S_n(f;x)- f(x)\right| & \leq & 2K_2\left(f,\dfrac{1}{4} \mu_{n,2}(x)\right).
\end{eqnarray*}
Using (\ref{E8}) and $\omega_2(f, \lambda \delta) \leq (\lambda + 1)^2 \omega_2(f, \delta)$ for $\lambda >0$, we get
\begin{eqnarray*}
\left|S_n(f;x)- f(x)\right| & \leq & C \omega_2\left(f, \sqrt{\mu_{n,2}(x)}\right),
\end{eqnarray*}
for some constant $C>0$.
\end{proof}

\subsection{A Voronovskaja-type result}
In this section we prove a Voronovskaja-type theorem for the operators $S_n$ given in (\ref{RBG1}).
\begin{lemma}\label{L6}
$\lim_{n \rightarrow \infty} b_n^2 \mu_{n,4}(x) = 3 x^2$ uniformly with respect to $x \in [0, a],\, a > 0$.
\end{lemma}
\begin{proof}
The result is obvious from lemma \ref{L5}(d).
\end{proof}

\begin{theorem}\label{T3}
For every $f \in C_\gamma[0,\infty)$ such that $f^{'}, f^{''} \in C_\gamma[0,\infty), \gamma \geq 4$, we have 
\begin{equation*}
\lim_{n \rightarrow \infty} b_n \left[S_n(f;x)-f(x) \right]=\dfrac{x}{2}f^{''}(x)
\end{equation*}
with respect to $x \in [0, a] \, (a > 0)$.
\end{theorem}

\begin{proof}
Let $f,f^{'}, f^{''} \in C_\gamma[0,\infty)$ and $x \geq 0$. Define
\begin{equation*}
\Psi(t,x)= \dfrac{f(t)-f(x)-(t-x)f^{'}(x)-\dfrac{1}{2}(t-x)^2 f^{''}(x)}{(t-x)^2}, \quad \text{if} \, t \neq x, 
\end{equation*}
and $\Psi(x,x)=0$. Then the function $\Psi(\cdot,x) \in C_\gamma[0,\infty)$. Hence, by Taylor's theorem we get
\begin{equation*}
f(t)= f(x)+(t-x)f^{'}(x)+\dfrac{1}{2}(t-x)^2 f^{''}(x)+(t-x)^2\Psi(t,x). 
\end{equation*}
Now from lemma(\ref{L5})(a)-(b)
\begin{eqnarray}\label{E12}
b_n \left[S_n(f;x)-f(x)\right]& = & \dfrac{1}{2} b_nf^{''}(x)\mu_{n,2}(x) + b_n S_{n}((t-x)^2 \Psi(t,x)).
\end{eqnarray}

If we apply the Cauchy-Schwarz inequality to the second term on the right hand side of (\ref{E12}), then 
\begin{equation*}
b_n S_{n}((t-x)^2 \Psi(t,x);x) \leq \left( b_n^2 \mu_{n,4}(x)\right)^{\frac{1}{2}} ( S_{n}(\Psi^2(t,x);x))^{\frac{1}{2}}
\end{equation*}
Now $\Psi^2(\cdot,x) \in C_\gamma[0,\infty)$, using theorem \ref{T1}, we have $S_{n}(\Psi^2(t,x);x) \rightarrow \Psi^2(x,x)=0$, as $n \rightarrow \infty$ and using lemma \ref{L6}, this third term on the right tends to zero for $x \in [0, a]$ and we get
\begin{equation*}
\lim_{n \rightarrow \infty} b_n \left[S_n(f;x)-f(x)\right]= \dfrac{1}{2} xf^{''}(x).
\end{equation*}
for $x \in [0, a], (a>0)$.
\end{proof}

\section{Polynomial weight spaces}
The results discussed in Section $1$, are of a local character dealing with compact sub-intervals of $[0, \infty)$. To derive global results for continuous functions on unbounded interval, one has to consider the spaces other than discussed earlier (refer \cite{Gad}). One such space is the polynomial weight space, where one can established global results for the operators (\ref{RBG1}) as well as a characterization for the same. So, in this section, we are discussing a polynomial weight space $C_N$ as given below and discuss direct and inverse results for the non-optimal cases $0 < \alpha < 2$, as well as for the saturation case $\alpha = 2$. \\
Consider the space $C_N$ defined using weight $w_N$, $N \in \mathbb{N}$ as follows:
\begin{eqnarray*}
&& w_0(x) = 1, \qquad w_N(x) = (1 + x^N)^{-1} \qquad (x \geq 0, N \in \mathbb{N}), \\ \\ &&
C_N = \left\lbrace f \in C[0, \infty): w_Nf \; \text{uniformly continous and bounded on} \; [0, \infty) \right\rbrace, \\ \\ &&
\Vert f \Vert_N = \sup_{x \geq 0} w_N(x)\left|f(x) \right|.
\end{eqnarray*}
The corresponding Leibschitz classes are given for $0 < \alpha \leq 2$ by $(h > 0)$
\begin{eqnarray*}
&& \Delta^2_h f(X) = f(x+2h) - 2f(x+h) + f(x), \\ \\ &&
\omega_N^2(f, \delta) = \sup_{0 < h \leq \delta} \Vert \Delta_h^2 f \Vert_N, \\ \\ &&
Lip_N^2 \alpha = \left\lbrace f \in C_N: \omega_N^2(f, \delta) =  O\left(\delta^{\alpha}\right), \delta \rightarrow 0^{+} \right\rbrace.
\end{eqnarray*}
We have the following characterization using the operators defined by (\ref{RBG1}):
\begin{theorem}\label{T4}
Let $f \in C_N, N \in \mathbb{N}$,  $\alpha \in (0, 2]$, then for the operators given by (\ref{RBG1}),
\begin{equation} \label{glo1}
w_N(x)\left| S_n(f;x) - f(x) \right| \leq M_N \left[ \dfrac{x}{b_n} \right]^{\alpha/2} \qquad (N \in \mathbb{N}, x \geq 0)
\end{equation}
is equivalent to $f \in Lip_N^2 \alpha$, where $M_N$ is a constant independent of $b_n$ and $x$.
\end{theorem}
\noindent Remark: The condition (\ref{glo1}) exhibit local structure. As the constant $M_N$ is independent of $b_n, x$, we may write (\ref{glo1}) as 
\begin{equation} \label{glo2}
\Vert x^{-\alpha/2} \left[ S_n(f;x) - f(x) \right] \Vert_N = O\left(b_n^{-\alpha/2} \right).
\end{equation}
which reveals the global equivalence assertion and this is due to incorporation of weights $w_N$ into the approximation (\ref{glo1}) as well as into the definition of $Lip_N^2 \alpha$.
\begin{lemma}\label{L7}
For each $N \in \mathbb{N}\cup \left\lbrace 0 \right\rbrace$ there is a constant $M_N$ such that uniformly for $x \geq 0$, $n \in \mathbb{N}$ 
\begin{equation}\label{glo3}
w_N(x)S_n(1/w_N(t);x)\leq M_N.
\end{equation}
In particular, for any $f \in C_N$,
\begin{equation}\label{glo4}
\Vert S_n(f; \cdot) \Vert_N \leq M_N \Vert f \Vert_N.
\end{equation}
\end{lemma}
\begin{proof}
For $N=0$, $w_0(x)S_n(1/w_0(t);x)= 1\leq M_N$, for any constant $M_N \geq 1$. Now for $N \in \mathbb{N}$,
\begin{eqnarray*}
w_N(x)S_n(1/w_N(t);x)&=& w_N(x)\left[S_n(1;x)+S_n(t^N;x)\right]
 \\ &=& \left(1+x^N\right)^{-1} \left[1+x^N+\sum_{j=1}^{N-1}a_{N,j}x^jb_n^{j-N} \right] \quad (\text{by lemma \ref{L4}})\\ &=& 1+ \dfrac{1}{b_n}\sum_{j=1}^{N-1}a_{N,j}x^jb_n^{j-(N-1)}/\left(1+x^N\right)\leq M_N
\end{eqnarray*}
as the coefficients $a_N,j$ only depend on $N$ and the sum is bounded with respect to $x, b_n$.\\ Now, since $f \in C_N$
\begin{eqnarray*}
w_N(x)\left|S_n(f(t);x)\right| &\leq & w_N(x) \sum_{k=0}^{\infty} w_N\left(\dfrac{k}{b_n}\right)f\left(\dfrac{k}{b_n}\right)\left[w_N\left(\dfrac{k}{b_n}\right)\right]^{-1}s_{b_n,k} \\ & \leq & \Vert f \Vert_N w_N(x) S_n(1/w_N(t);x)\leq M_N \Vert f \Vert_N.
\end{eqnarray*}
Taking supremum on right hand side over all $x$, we get the result.
\end{proof}
\begin{lemma}\label{L8}
For each $N \in \mathbb{N}\cup \left\lbrace 0 \right\rbrace$ there is a constant $M_N$ such that for all $x \geq 0$, $n \in \mathbb{N}$
\begin{equation}\label{glo5}
w_N(x) S_n((t-x)^2/w_N(t);x) \leq M_Nx/b_n.
\end{equation}
Furthermore, one has 
\begin{equation}\label{glo6}
\dfrac{1}{b_n} w_N(x) S_n(t/w_N(t);x) \leq M_Nx/b_n.
\end{equation}
\end{lemma}
\begin{proof}
For $N=0$, by (c) of lemma \ref{L5}, $w_0(x) S_n((t-x)^2/w_0(t);x)= S_n((t-x)^2;x)= x/b_n \leq M_N x/b_n$ for a constant $M_N \geq 1$. \\ For $N=1$, by (c)-(d) of lemma \ref{L5},
\begin{eqnarray*}
w_1(x) S_n((t-x)^2/w_1(t);x) &=& w_1(x)\left[S_n((t-x)^2;x)+ S_n((t-x)^3;x)\right. \\ && \qquad \left. +x S_n((t-x)^2;x)\right] \\ &=& \dfrac{x}{b_n}\left[1 + \dfrac{1}{b_n(1+x)} \right] \leq M_N x/b_n
\end{eqnarray*}
for a constant $M_N \geq 2$. \\ For $N \geq 2$, lemma \ref{L4} and (\ref{e24}) imply
\begin{eqnarray*}
S_n((t-x)^2t^N;x) &=& S_n(t^{N+2};x)-2x S_n(t^{N+1};x)+x^2 S_n(t^{N};x)\\ &=& x^{N+2} + a_{N+2,N+1}x^{N+1}/b_n+ \cdots + xb_n^{-N-1} \\ && \quad -2x \left[ x^{N+1} + a_{N+1,N}x^{N}/b_n+ \cdots + xb_n^{-N}\right] \\ && \quad +x^2 \left[ x^{N} + a_{N,N-1}x^{N-1}/b_n+ \cdots + xb_n^{-N+1} \right] \\ &=& \left(a_{N+2,N+1}-2a_{N+1,N}+a_{N,N-1} \right)x^{N+1}/b_n + \cdots xb_n^{-N-1} \\ &=& \left[x^N + \cdots + b_n^{-N} \right]\dfrac{x}{b_n}.
\end{eqnarray*}
Hence
\begin{eqnarray*}
w_N(x)S_n((t-x)^2/w_N(t);x) &=& w_N(x) \left[1+ x^N + \cdots + b_n^{-N} \right]\dfrac{x}{b_n} \leq M_N \dfrac{x}{b_n}.
\end{eqnarray*}
This proves (\ref{glo5}). Also
\begin{eqnarray*}
S_n(t/w_N(t);x) &=& S_n(t;x)+ S_n(t^{N+1};x) \\ &=&  x+ x^{N+1} + \cdots +x b_n^{-N}\\ &=&  \left[1+ x^{N} + \cdots + b_n^{-N} \right]x .
\end{eqnarray*}
Therefore, 
\begin{eqnarray*}
\dfrac{1}{b_n}w_N(x) S_n(t/w_N(t);x) &=& w_N(x) \left[1+ x^{N} + \cdots + b_n^{-N} \right]\dfrac{x}{b_n} \leq M_N \dfrac{x}{b_n}.
\end{eqnarray*}
\end{proof}
\subsection{Direct result}
\begin{lemma}\label{L9}
Let $N \in \mathbb{N}\cup \left\lbrace 0 \right\rbrace$, $g \in C_N^2 := \left\lbrace f \in C_N : f^{'}, f^{''} \in C_N \right\rbrace$. Then there exists a constant $M_N$ such that for all $n \in \mathbb{N}$, $x \geq 0$
\begin{equation}\label{glo7}
w_N(x)\left| S_n(g;x) - g(x) \right| \leq M_N \Vert g^{''} \Vert_N \dfrac{x}{b_n}.
\end{equation}
\end{lemma}
\begin{proof}
Using Taylor's theorem, we can write
\begin{equation*}
g(t) - g(x) = (t-x)g^{'}(x) + \int_x^t \int_x^s g^{''}(u) du ds \qquad \qquad (x, t \geq 0)
\end{equation*}
and the estimate 
\begin{eqnarray*}
\left|\int_x^t \int_x^s \right| g^{''}(u) \left| du ds \right| &\leq & \Vert g^{''}\Vert_N \left| \int_x^t \int_x^s  \dfrac{1}{w_N(u)} du ds \right| \\ &\leq & (1/2) \Vert g^{''}\Vert_N(t-x)^2\left[ \dfrac{1}{w_N(x)} + \dfrac{1}{w_N(t)}\right],
\end{eqnarray*}
the proof follows since by lemma \ref{L3}(a)-(b), lemma \ref{L5}(b) and (\ref{glo5})
\begin{eqnarray*}
w_N(x)\left|S_n(g ;x) - g(x) \right| &\leq & (1/2) \Vert g^{''}\Vert_N \left[\mu_{n,2}(x) + w_N(x) S_n\left(\dfrac{(t-x)^2}{w_N(t)};x \right) \right] \\ &\leq & M_N \Vert g^{''} \Vert_N \dfrac{x}{b_n}.
\end{eqnarray*}
\end{proof}
Now we can write the proof of direct part of theorem \ref{T4}.
\begin{theorem}\label{T5}
For any $N \in \mathbb{N} \cup \left\lbrace 0 \right\rbrace$, $f \in C_N$ there holds for all $x \geq 0$, $n \in \mathbb{N}$
\begin{equation}\label{glo8}
w_N(x) \left| S_n(f;x) - f(x) \right| \leq M_N \omega_N^2\left( f, \sqrt{\dfrac{x}{b_n}}\right).
\end{equation}
In particular, if $f \in Lip_N^2 \alpha$ for some $\alpha \in (0, 2]$, then 
\begin{equation*}
w_N(x) \left| S_n(f;x) - f(x) \right| \leq M_N \left[\dfrac{x}{b_n}\right]^{\alpha/2}.
\end{equation*}
\end{theorem}
\begin{proof}
For $N=0$, the assertion follows from theorem \ref{T2}. For $f \in C_N$, $h>0$ one has by lemma \ref{L7}, \ref{L9} and by (\ref{e3}) that 
\begin{eqnarray*}
w_N(x)\left|S_n(f ;x) - f(x) \right| &\leq & w_N(x)\left|S_n([f-f_h] ;x) \right| \\ && \quad + w_N(x)\left|S_n(f_h ;x) - f_h(x) \right| + w_N(x) \left|f_h(x) - f(x) \right| \\ &\leq & \Vert f - f_h \Vert_N \left[w_N(x) S_n(1/w_N(t);x)+1\right] + M_N \Vert f_h^{''}\Vert_N \dfrac{x}{b_n} \\ &\leq &  M_N \omega_N^2(f,h)\left[1 + \dfrac{x}{b_nh^2} \right],
\end{eqnarray*}
so that (\ref{e3}) follows upon setting $h = \sqrt{\dfrac{x}{b_n}}$.
\end{proof}
Remark: The above theorem implies that for each $f \in C_N$, $x \geq 0$
\begin{eqnarray*}
\lim_{n \rightarrow \infty }w_N(x)\left|S_n(f ;x) - f(x) \right|=0.
\end{eqnarray*}

\subsection{Inverse result}
The main tool for the proof of the inverse theorems in non-optimal case $0<\alpha<2$ is an appropriate Bernstein-type inequality.
\begin{lemma}\label{L10}
For $f \in C_N$ and $x, \delta >0$ there holds
\begin{equation}\label{glo9}
w_N(x)\left| S_n^{''}(f;x)\right| \leq M_N \omega^2_N (f, \delta) \left[\dfrac{b_n}{x} + \delta^{-2} \right].
\end{equation}
\end{lemma}
\begin{proof}
We can represent $S_n^{''}(f;x)$ in two ways:
\begin{equation}\label{glo10}
S_n^{''}(f;x)= \left[\dfrac{b_n}{x}\right]^2\sum_{k=0}^{\infty}r_{n,k}(x)f\left(\dfrac{k}{b_n}\right)s_{b_n,k}(x),
\end{equation}
where $r_{n,k}(x)=\left[\left(\dfrac{k}{b_n}-x \right)^2 - \dfrac{k}{b_n^2} \right]$ and 
\begin{equation}\label{glo11}
S_n^{''}(f;x)= b_n^2\sum_{k=0}^{\infty}\Delta_{1/b_n}^2 f\left(\dfrac{k}{b_n}\right)s_{b_n,k}(x).
\end{equation}
From (\ref{e3}) there follows
\begin{eqnarray}\label{glo12}
\left|\Delta_{1/b_n}^2 f_{\delta}\left(\dfrac{k}{b_n}\right)\right| &\leq & \int_0^{1/b_n}\int_0^{1/b_n}\left|f_{\delta}^{''}(t+u+v) \right|du dv \nonumber \\ & \leq & \Vert f_{\delta}^{''}\Vert_N\int_0^{1/b_n}\int_0^{1/b_n}\dfrac{1}{w_N(t+u+v)} du dv \nonumber \\ & \leq & 9 (b_n\delta)^{-2}\omega_N^2(f, \delta)/w_N(t+ 2/b_n).
\end{eqnarray}
By (\ref{glo10}), (\ref{glo11}) and (\ref{glo12}) one has for $x, \delta >0$ 
\begin{eqnarray*}
w_N(x) \left|S_n^{''}(f;x) \right| &\leq & w_N(x) \left|S_n^{''}\left(f - f_{\delta};x \right) \right| + w_N(x) \left|S_n^{''}\left(f_{\delta};x\right) \right| \\ & \leq & w_N(x) (b_n/x)^2 \sum_{k=0}^{\infty}\left| r_{n,k}(x)\right| \left| f\left(\dfrac{k}{b_n} \right) - f_{\delta}\left(\dfrac{k}{b_n} \right) \right| s_{b_n,k}(x) \\ && \quad + w_N(x) b_n^2 \sum_{k=0}^{\infty} \left|\Delta_{1/b_n}^2 f_{\delta}\left( \dfrac{k}{b_n}\right)\right| s_{b_n,k}(x)\\ & \leq & \omega_N^2(f, \delta)\left\lbrace w_N(x) (b_n/x)^2 \sum_{k=0}^{\infty}\left| r_{n,k}(x)\right| \left[w_N\left(\dfrac{k}{b_n} \right) \right]^{-1}s_{b_n,k}(x) \right. \\ && \left. \quad  + 18 \delta^{-2} w_N(x)  \sum_{k=0}^{\infty} \left[w_N\left(\dfrac{k+2}{b_n} \right) \right]^{-1}s_{b_n,k}(x)\right\rbrace.
\end{eqnarray*}
To obtain (\ref{glo9}) we have to show that 
\begin{equation}\label{glo13}
w_N(x) (b_n/x)^2 \sum_{k=0}^{\infty}\left| r_{n,k}(x)\right| \left[w_N\left(\dfrac{k}{b_n} \right) \right]^{-1}s_{b_n,k}(x) \leq M_N \dfrac{x}{b_n} 
\end{equation}
and
\begin{equation}\label{glo14}
w_N(x)  \sum_{k=0}^{\infty} \left[w_N\left(\dfrac{k+2}{b_n} \right) \right]^{-1}s_{b_{\bar{n}},k}(x)\leq M_N.
\end{equation}
Now (\ref{glo14}) follows from (\ref{glo3}) since there exists a constant $M_N^{'}$ such that for all $k \in \mathbb{N} \cup {0}$ , $n \in \mathbb{N}$
\begin{equation*}
\left[w_N\left( \dfrac{k+2}{b_n}\right)\right]^{-1}\leq M_N^{'}\left[w_N(k/b_{\bar{n}})\right]^{-1},
\end{equation*}
whereas (\ref{glo13}) follows from lemma \ref{L7}, \ref{L8} since
\begin{equation*}
w_N(x)\left[S_n((t-x)^2/w_N(t);x) + \dfrac{1}{b_n}S_n(t/w_N(t);x)\right]\leq M_N\dfrac{x}{b_n}.
\end{equation*}
\end{proof}
\begin{lemma}[\cite{Bec},lemma 10]\label{L11}
For $x \geq 0$, $0 < h \leq 1$
\begin{equation}\label{glo15}
\int_0^h \int_0^h \dfrac{ds dt}{x + s + t } \leq \dfrac{6h^2}{x + 2h}.
\end{equation}
\end{lemma}
\begin{proof}
One has
\begin{equation*}
\int_0^h \int_0^h \dfrac{ds dt}{x + s + t } = \Delta_h^2[t \log t](x),
\end{equation*}
which in particular hold true for $x=0$. This yields for $x \in [0,h]$
\begin{equation*}
\int_0^h \int_0^h \dfrac{ds dt}{x + s + t } \leq \int_0^h \int_0^h \dfrac{ds dt}{ s + t } = 2h \log 2 < \dfrac{6h^2}{3h}\leq \dfrac{6h^2}{x+2h}.
\end{equation*}
For $h \leq x$ one has 
\begin{equation*}
\int_0^x\int_0^x \dfrac{ds dt}{x + s + t } \leq  \dfrac{3h^2}{3x}\leq \dfrac{3h^2}{x+2h}.
\end{equation*}
This proves (\ref{glo15}).
\end{proof}
\begin{theorem}\label{T6}
Let $N \in \mathbb{N} \cup \left\lbrace 0 \right\rbrace$. If $f \in C_N$ satisfies for some $\alpha \in (0, 2)$ and all $N \in \mathbb{N}$, $x \geq 0$
\begin{equation}\label{glo16}
w_N(x)\left|S_n(f;x) - f(x) \right| \leq M_N \left[ \dfrac{x}{b_n} \right]^{\alpha/2},
\end{equation}
then $f \in Lip_N^2 \alpha$.
\end{theorem}
\begin{proof}
It is sufficient to show that for $0< h$, $\delta \leq 1$, $\delta < \sqrt{h}$ 
\begin{equation}\label{glo17}
\omega_N^2(f,h) \leq M_N \left[\delta^{\alpha} + (h/\delta)^2\omega_N^2(f,\delta) \right]. 
\end{equation}
To this end, fix $0 < h$, $\delta \leq 1$, $\delta < \sqrt{h}$, $x \geq 0$. Using lemmas (\ref{L10}), (\ref{L11}) and observing that $w_N(x)/w_N(x+2h)\leq 3^N$ as $h \leq 1$, there follows from (\ref{glo16}) for all $n \in \mathbb{N}$
\begin{eqnarray*}
w_N(x)\left|\Delta_h^2f(x) \right| & \leq & w_N(x)\left|f(x+2h) - S_n(f(t + 2h);x) \right| \\ && + 2w_N(x)\left|S_n(f(t + h);x) - f(x+h) \right| \\ && + w_N(x)\left|f(x) - S_n(f(t);x) \right| + w_N(x)\left|\Delta_h^2S_n(f(t);x) \right| \\ & \leq & M_N \left[ \dfrac{x + 2h}{b_n} \right]^{\alpha/2} \left\lbrace \dfrac{w_N(x)}{w_N(x + 2h)} + \dfrac{2w_N(x)}{w_N(x + h)} + 1 \right\rbrace \\ && + w_N(x) \int_0^h \int_0^h\left| S_n^{''}(t+s+u;x)\right| ds du \\ & \leq & M_N \left[ \dfrac{x + 2h}{b_n} \right]^{\alpha/2} + M_N\omega_N^2(f,\delta)\left\lbrace \dfrac{w_N(x)}{w_N(x + 2h)} \right\rbrace \\ && \left[b_n\int_0^h \int_0^h \dfrac{ds du}{x + s+ u} + \left(\dfrac{h}{\delta} \right)^2 \right] \\ & \leq & M_N \left\lbrace \left[ \dfrac{x + 2h}{b_n} \right]^{\alpha/2} + \left[M b_n \dfrac{h^2}{x+2h} + \left(\dfrac{h}{\delta} \right)^2 \right]\omega_N^2(f,\delta) \right\rbrace.
\end{eqnarray*}
For the case $x=0$ let us only note that the estimate holds true in view of the existence of the integral for $x=0$ and the continuity of the expressions involved. Now choose $b_n$ such that 
\begin{equation*}
\sqrt{\dfrac{x+2h}{b_n}} \leq \delta < \sqrt{\dfrac{x+2h}{(b_n - 1)}} \leq \sqrt{2} \sqrt{\dfrac{x+2h}{b_n}},
\end{equation*}
the last expression being $\geq 2 \sqrt{\dfrac{h}{b_n}} $. Then 
\begin{equation*}
\omega_N^2(f,h) \leq M_N \left[\delta^{\alpha} + (h/\delta)^2\omega_N^2(f,\delta) \right]. 
\end{equation*}
proving (\ref{glo17}). This completes the proof.
\end{proof}

\begin{theorem}\label{T7}
Let $N \in \mathbb{N} \cup \left\lbrace 0 \right\rbrace$. If $f \in C_N$ satisfies for all $N \in \mathbb{N}$, $x \geq 0$
\begin{equation}\label{glo18}
w_N(x)\left|S_n(f;x) - f(x) \right| \leq M_N \left[ \dfrac{x}{b_n} \right],
\end{equation}
then $f \in Lip_N^2 2$.
\end{theorem}
\begin{proof}
For any $f \in C_N$ one has
\begin{equation}\label{glo19}
f \quad \text{is convex on} \quad [0, \infty) \quad \text{iff} \quad S_n(f;x) \geq f(x) \quad \text{for all} \quad n \in \mathbb{N}, x \geq 0.
\end{equation}
Furthermore, the representation (\ref{e23}) implies 
\begin{eqnarray*}
S_n(t^{N+2};x) & = & x^{N+2} + \dfrac{(N+2)(N+1)}{2b_n}x^{N+1}+ \cdot + \dfrac{x}{b_n^{N+1}} \\ & \geq & x^{N+2} + \dfrac{x^{N+1}}{b_n} = x^{N}S_n(t^2;x)
\end{eqnarray*}
for all $N \in \mathbb{N} \cup {0}$, so that
\begin{equation}\label{glo20}
S_n(t^2;x)/w_N(x) \leq S_n(t^2/w_N(t);x).
\end{equation}
In view of lemma \ref{L3}(a)-(b), lemma \ref{L5}(b) and equation (\ref{glo18}), for $f \in C_N$
\begin{equation*}
\pm f(x) \leq \dfrac{M_N}{w_N(x)} \left( S_n(t^2;x) - x^2 \right) + S_n(\pm f(t); x).
\end{equation*}
Then one has by (\ref{glo20})
\begin{eqnarray*}
\dfrac{M_N x^2}{w_N(x)} \pm f(x) & \leq & M_N \dfrac{S_n(t^2;x)}{w_N(x)} + S_n(\pm f(t);x) \\ & \leq & S_n \left(\dfrac{M_Nt^2}{w_N(t)} \pm f(t);x \right).
\end{eqnarray*}
Therefore $\dfrac{M_N t^2}{w_N(t)} \pm f(t)$ are convex functions on $[0, \infty)$, that is,
\begin{equation*}
\Delta_h^2\left[\dfrac{M_N t^2}{w_N(t)} \pm f(t)\right](x) \geq 0.
\end{equation*}
In other words
\begin{equation*}
w_N(x)\left|\Delta_h^2 f(x) \right| \leq M_Nw_N(x)\Delta_h^2 \left[\dfrac{t^2}{w_N(t)}\right](x) \leq M_N h^2.
\end{equation*}
This proves that $f \in Lip_N^2 2 $.
\end{proof}

\section{Examples, graphs and analysis}
The modification (\ref{RBG1}) of (\ref{Mir}) is due to introduction of sequence $\left( b_n \right)_{1}^{\infty}$ which is an increasing sequence of real numbers with $b_n \rightarrow \infty$ as $n \rightarrow \infty$ and $b_1 \geq 1$. There are many such sequences in existence. Consider, for example, $b_n = r^n$, with $r > 1$ or take $b_n = n^m$ with $m \geq 1$. In case of these examples, there is an $N$ for which $b_n > n$ for all $n \geq N$. If we consider for $b_n$, $b_n = \sum_{N=1}^{n} \dfrac{1}{N^p}$, where $0 \leq p \leq 1$, then one can easily check that the sequence $\left( b_n \right)$ is increasing sequence with $b_1 \geq 1$, $b_n \rightarrow \infty$ as $n \rightarrow \infty$ and $b_n \leq n$ for all $n$. These sequences are well controlled and ordered sequences as in these cases, we get $ \sum_{N=1}^{n} \frac{1}{N} < \sum_{N=1}^{n} \frac{1}{N^{p_1}} < \sum_{N=1}^{n} \frac{1}{N^{p_2}} < n $ for $ 0 < p_2 < p_1<1 $. Some graphs of functions under the operators $M_n$ and $S_n$ represented by equations (\ref{Mir}) and (\ref{RBG1}), respectively, are given below showing how the choices of sequence $\left( b_n \right)$ affects the approximation procedure by the operators. As these operators are representations in the form of infinite series, we consider truncation of the series by fixing the values of summation index $k$. \\
Figures (\ref{F1a}-\ref{F1c}), represent graphs of $f(x) = e^x$ (Magenta), graphs of $M_n(f;x)$ (Blue), graphs of $S_n(f;x)$ where $b_n = \sum_1^n (1/N)$ (Red) and $b_n = \sum_1^n(1/\sqrt{N})$ (Green) for $n = 3 \, \text{to} \, 15$, $k = 50$ corresponding to $x \in [0,2]$, $x \in [0,4]$ and $x \in [0, 6]$, respectively. It can be observed that as interval expands the graphs of better estimate get degenerate first. Figures (\ref{F2a}-\ref{F2c}), represent the same graphs for $k = 100$, It can be seen that the effect of degeneration is now delayed. Figures (\ref{F4a}-\ref{F4c}), represent the same graphs for $k = 100$ but for $n = 20 \, \text{to} \, 30$, where more consolidation of graphs can be seen. Figure (\ref{F5b}) plots $M_n(f;x)$ (Blue) for $n = 20 \, \text{to} \, 30$, while $S_n(f;x)$ where $b_n = \sum_1^n (1/N)$ (Red) and $b_n = \sum_1^n(1/\sqrt{N})$ (Green) for $n = 120 \, \text{to} \, 130$. The graphs of $S_n(f;x)$ where $b_n = \sum_1^n(1/\sqrt{N})$ (Green) for $n = 120 \, \text{to} \, 130$ seem to be achieving the same accuracy as $M_n(f;x)$ for $n = 20 \, \text{to} \, 30$. Figure (\ref{F5c}) plots $M_n(f;x)$ for $n =7, 8, 9$ and $S_n(f;x)$ where $b_n = \sum_1^n(1/\sqrt{N})$ (Red) for $n = 78 \, \text{to} \, 95$. We have more consolidation for $S_n(f;x)$ then for $M_n(f;x)$. Figures (\ref{F6a}-\ref{F6b}) show similar comparison for $f(x) = \sin x$ (Magenta), $x \in [0, 2\pi]$, $k = 100$ and $k = 120$, respectively, for $M_n(f;x)$ (Blue) for $n = 20 \, \text{to} \, 25$, $S_n(f;x)$ where $b_n = \sum_1^n(1/\sqrt{N})$ (Red) for $n = 80 \, \text{to} \, 100$ and $S_n(f;x)$ where $b_n = \sum_1^n(1/N)$ (Green) for $n = 80 \, \text{to} \, 100$.
\begin{figure}
\centering
\begin{subfigure}[b]{0.3\textwidth}
\includegraphics[width=\textwidth]{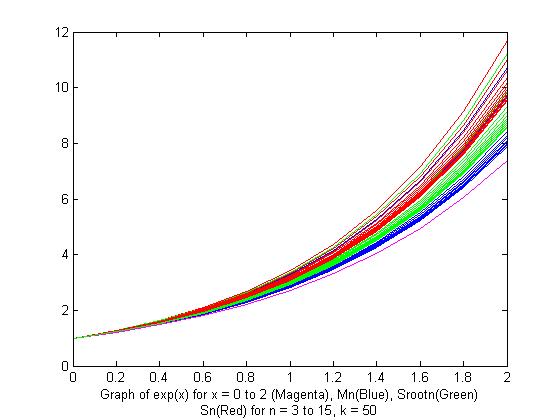}
\caption{}
\label{F1a}
\end{subfigure}
\begin{subfigure}[b]{0.3\textwidth}
\centering
\includegraphics[width=\textwidth]{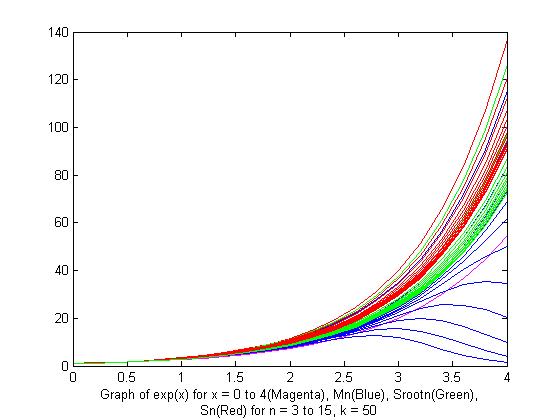}
\caption{}
\label{F1b}
\end{subfigure}
\begin{subfigure}[b]{0.3\textwidth}
\centering
\includegraphics[width=\textwidth]{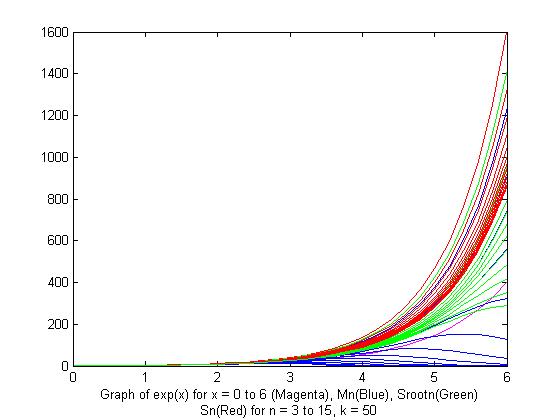}
\caption{}
\label{F1c}
\end{subfigure}
\caption{}
\label{F1}
\end{figure}
\begin{figure}
\centering
\begin{subfigure}[b]{0.3\textwidth}
\includegraphics[width=\textwidth]{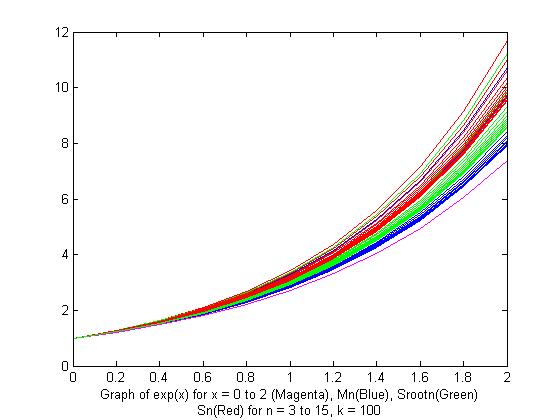}
\caption{}
\label{F2a}
\end{subfigure}
\begin{subfigure}[b]{0.3\textwidth}
\centering
\includegraphics[width=\textwidth]{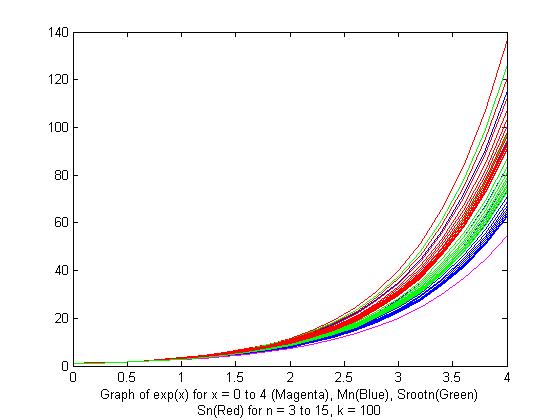}
\caption{}
\label{F2b}
\end{subfigure}
\begin{subfigure}[b]{0.3\textwidth}
\centering
\includegraphics[width=\textwidth]{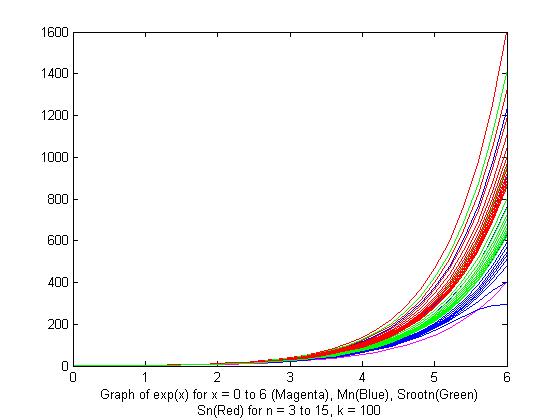}
\caption{}
\label{F2c}
\end{subfigure}
\caption{}
\label{F2}
\end{figure}
\begin{figure}
\centering
\begin{subfigure}[b]{0.3\textwidth}
\includegraphics[width=\textwidth]{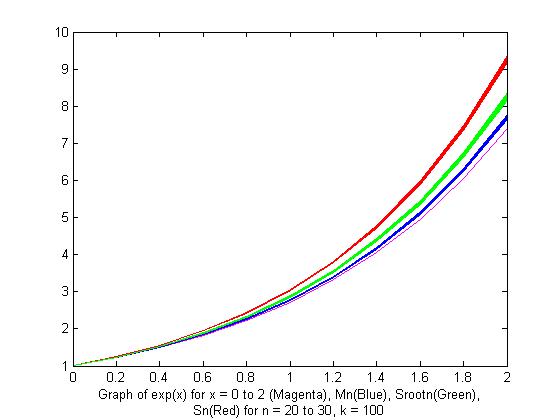}
\caption{}
\label{F4a}
\end{subfigure}
\begin{subfigure}[b]{0.3\textwidth}
\centering
\includegraphics[width=\textwidth]{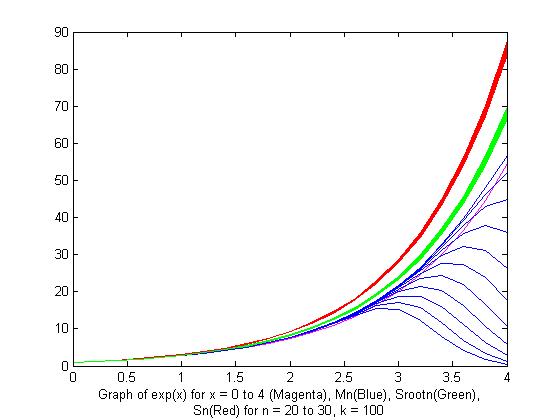}
\caption{}
\label{F4b}
\end{subfigure}
\begin{subfigure}[b]{0.3\textwidth}
\centering
\includegraphics[width=\textwidth]{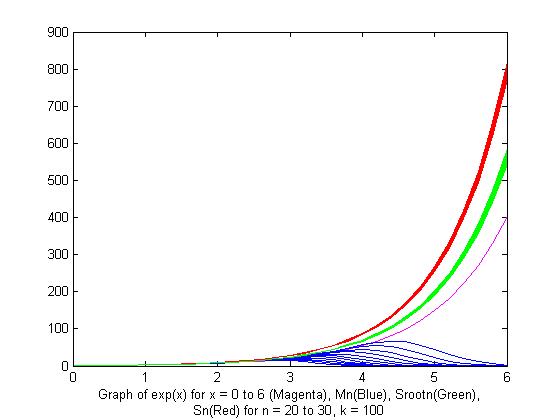}
\caption{}
\label{F4c}
\end{subfigure}
\caption{}
\label{F4}
\end{figure}
\begin{figure}[h]
\centering
\begin{subfigure}[b]{0.49\textwidth}
\includegraphics[width=\textwidth]{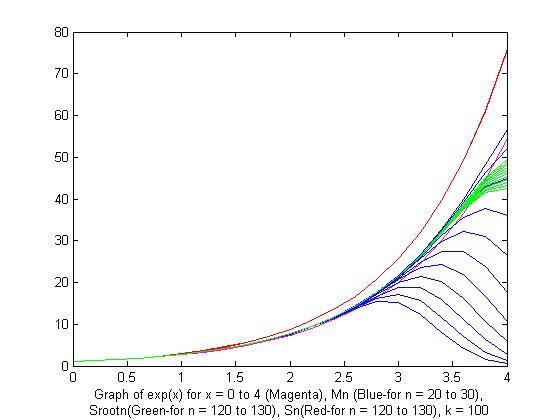}
\caption{}
\label{F5b}
\end{subfigure}
\begin{subfigure}[b]{0.49\textwidth}
\includegraphics[width=\textwidth]{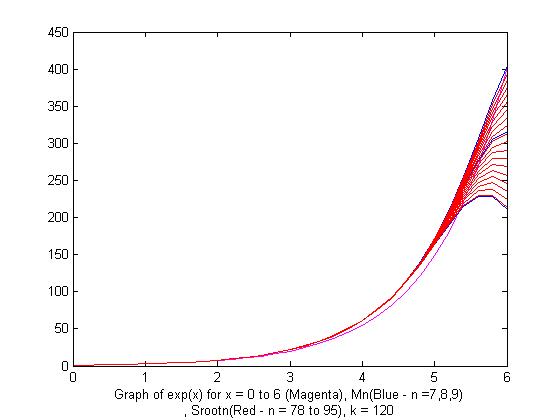}
\caption{}
\label{F5c}
\end{subfigure}
\caption{}
\end{figure}
\label{F5}
\begin{figure}
\centering
\begin{subfigure}[b]{0.49\textwidth}
\centering
\includegraphics[width=\textwidth]{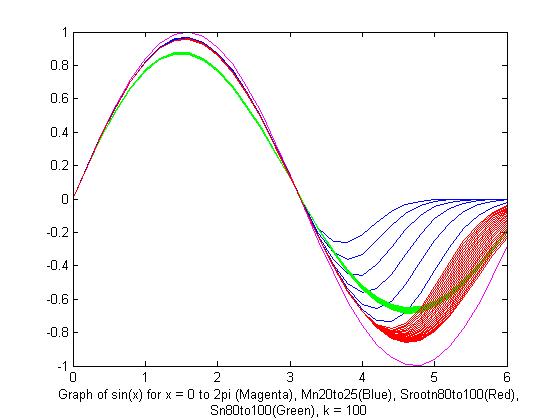}
\caption{}
\label{F6a}
\end{subfigure}
\begin{subfigure}[b]{0.49\textwidth}
\centering
\includegraphics[width=\textwidth]{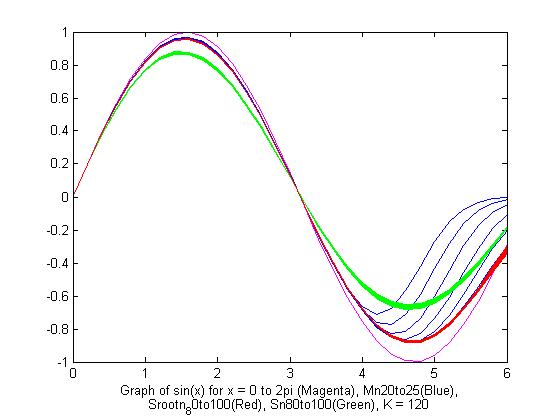}
\caption{}
\label{F6b}
\end{subfigure}
\caption{}
\label{F6}
\end{figure}

\end{document}